\documentclass{amsart}
\usepackage{tikz}
\usetikzlibrary{shapes.geometric}
\usepackage[center]{caption}
\usepackage{xcolor}
\usepackage{amssymb,latexsym,amsmath,extarrows}
\usepackage{graphicx,mathrsfs,comment}
\usepackage{hyperref,url}
\usepackage{pict2e}

\usepackage{amstext}
\usepackage{bbm}

\numberwithin{equation}{section}

\usepackage{esint}

\newtheorem{theorem}{Theorem}[section]
\newtheorem{lemma}[theorem]{Lemma}

\newtheorem{remark}[theorem]{Remark}

\newtheorem{corollary}[theorem]{Corollary}

\newtheorem{conjecture}[theorem]{Conjecture}

\makeatletter
\newcommand{\barredsum}{%
  \DOTSB\mathop{\mathpalette\@barredsum\relax}\slimits@
}
\newcommand{\@barredsum}[2]{%
  \begingroup
  \sbox\z@{$#1\sum$}%
  \setlength{\unitlength}{\dimexpr2pt+\ht\z@+\dp\z@\relax}%
  \@barredsumthickness{#1}%
  \vphantom{\@barredsumbar}%
  \ooalign{$\m@th#1\sum$\cr\hidewidth$#1\@barredsumbar$\hidewidth\cr}%
  \endgroup
}
\newcommand{\@barredsumbar}{%
  \vcenter{\hbox{\begin{picture}(0,1)\roundcap\Line(0,0)(0,1)\end{picture}}}%
}
\newcommand{\@barredsumthickness}[1]{
  \linethickness{%
    1.25\fontdimen8
      \ifx#1\displaystyle\textfont\else
      \ifx#1\textstyle\textfont\else
      \ifx#1\scriptstyle\scriptfont\else
      \scriptscriptfont\fi\fi\fi 3
  }%
}
\makeatother

\newcommand{\ZFp}{\mathbb{F}_p}

\newcommand{\de}{\delta}

\newcommand{\e}{\varepsilon}

\newcommand{\La}{\Lambda}
\newcommand{\si}{\sigma}

\newcommand{\wh}{\widehat}

\newcommand{\ZC}{\mathbb{C}}

\newcommand{\ZF}{\mathbb{F}}

 \newcommand{\pt}[1]{\left(#1 \right)}
 \newcommand{\vpt}[1]{\left\lvert #1 \right\rvert}
 
\begin{document}

\title{A bilinear estimate in $\ZF_p$}

\author{Necef Kavrut} \address{ Necef Kavrut\\  Department of Mathematics\\ California Institute of Technology, USA} \email{nkavrut@caltech.edu}

\author{Shukun Wu} \address{ Shukun Wu\\  Department of Mathematics\\ Indiana University Bloomington, USA} \email{shukwu@iu.edu}

\date{}
\begin{abstract}
We improve an $L^2\times L^2\to L^2$ estimate for a certain bilinear operator in the finite field of size $p$, where $p$ is a prime sufficiently large. Our method carefully picks the variables to apply the Cauchy-Schwarz inequality. As a corollary, we show that there exists a quadratic progression $x,x+y,x+y^2$ for nonzero $y$ inside any subset of $\ZF_p$ of density $\gtrsim p^{-1/8}$.

\end{abstract}

\maketitle

\section{Introduction}
Let $\ZFp$ be the finite field with $p$ elements ($p$ is a prime). We use the following convention throughout the paper: $e_p(x):= e^{-2\pi i\frac{ x}{p}}$ and 
\begin{align*}
    \hat{f}(z) &= \sum_{x\in \ZFp} f(x) e_p(xz)\\
    f(x) &= \frac{1}{p}\sum_{z\in \ZFp} \hat{f}(z) e_p(-xz)\\
    \|f\|_r &= \Big(\sum_{x\in \ZFp} \vpt{f(x)}^r \Big)^{1/r}\\
    \|f\|_{2} &= p^{-1}\|\hat{f}\|_{2} &&\text{(Parseval)}.
\end{align*}
The definitions here are exactly those in \cite{Bourgain-Chang, Dong-Li-Sawin} if we interchange between $\wh f$ and $f$, and if $e_p$ denotes $e^{2\pi i\frac{x}{p}}$ instead.

\smallskip

In this paper, we are interested in the following bilinear operator in $\ZF_p$:
\begin{align}
\label{operator}
    T (f_1,f_2)(s) = \sum_{n\not=s,\, n\in\ZF_p} f_1(s-n) f_2(n) K(s-n,n)
\end{align}
where the kernel $K$ is given as
\begin{equation}
\label{K}
    K(a,b)=\frac{1}{p}\sum_{y\in\ZF_p}e_p\pt{ay^2+by}.
\end{equation}
Note that $K(a,b)$ is a quadratic Gauss sum and can be evaluated via the well-known formula (see for instance \cite{Iwaniec-Kowalski}):
\begin{align}
\label{Gauss-sum}
    K(a,b)= \begin{cases}
        1 \quad &\text{if $a=b=0$}\\[1ex]
        0 \quad &\text{if $a=0$ but $b\neq 0$}\\
        p^{-1/2}\pt{\dfrac{a}{p}}e_p(-b^2 \overline{4a}) \sigma_p \quad &\text{if $a\neq 0$}
    \end{cases} 
\end{align}
where $\pt{\frac{\cdot}{p}}$ is the Legendre symbol, $a\bar a\equiv1\mod p$, and $\si_p$ with $|\sigma_p|=1$ depends only on $p$, hence can be ignored.

Regarding the $L^2$-boundedness of the operator $T$, Bourgain-Chang \cite{Bourgain-Chang} raised the following conjecture: 
\begin{conjecture}
For any $\e>0$, there exists $C_\e$ so that for large enough prime $p$,
\begin{equation}
\label{B-C-conj}
    \|T(f_1,f_2)\|_2\leq C_\e p^{\e-1/2}\|f_1\|_2\|f_2\|_2.
\end{equation}

\end{conjecture}

\noindent It is necessary that the summation in \eqref{operator} excludes $n=s$. Otherwise, by taking $f_1(n)=f_2(n)=\de_0(n)$ one gets $\|T(f_1,f_2)\|_2=1=\|f_1\|_2\|f_2\|_2$.

\smallskip

Bourgain-Chang \cite{Bourgain-Chang} first proved \eqref{B-C-conj} with the decay factor $p^{-1/10}$. This was later improved by Dong-Li-Sawin \cite{Dong-Li-Sawin} to the exponent $p^{-1/8}$, which, to some extent, is the limit of both methods in \cite{Bourgain-Chang} and \cite{Dong-Li-Sawin}.  In this paper, we carefully pick the variables to use the Cauchy-Schwarz inequality and get

\begin{theorem}
\label{main}
Given $f_1,f_2: \ZFp :\to \ZC$ one has
\begin{align}
\label{main-esti}
    \|T(f_1,f_2)\|_2 \lesssim p^{-3/16} \|f_1\|_2 \|f_2\|_2.
\end{align}
\end{theorem}

As an immediate corollary (see \cite{Bourgain-Chang}), we have the following Roth-type estimate:
\begin{corollary}
For any $|A|\subseteq \ZFp$ with $|A|=\delta p$ and $\delta \gtrsim p^{-1/8}$, there are $\gtrsim p^{13/8}$ triplets $x,x+y,x+y^2 \in A$.
\end{corollary}
\noindent See also \cite{Peluse} for another approach to attack the polynomial Roth-type problems, and \cite{Bourgain, Chen-Guo-Li} for similar results on the real line.

\smallskip

To prove Theorem \ref{main}, we need the following deep theorem about multidimensional exponential sum:

\begin{theorem}[\cite{Dong-Li-Sawin} Theorem 3.1]
\label{donglisawin-thm}
Let $F, G\in\ZF_p[X_1, \ldots, X_4]$. Assume that the degree of $F$ is indivisible by $p$, the homogeneous leading term of $G$ defines a smooth projective hypersurface, and the homogeneous leading terms of $G$ and that of $F$ together define a smooth co-dimension-2
variety in the projective space. Then
\begin{equation}
    \sum_{\substack{z_1,\ldots,z_4,\\G=0}}e_p(F)\ll p^{3/2}.
\end{equation}
\end{theorem}

\begin{remark}
\rm

With a pure analytic method, we can prove Theorem \ref{main} for a weaker exponent, which still leads to an improvement over the previous result in \cite{Dong-Li-Sawin}. This will be discussed in the ending remark.

\end{remark}

\medskip

Throughout the paper, we use $a\lesssim b$ to represent $a\leq Cb$ for some unimportant constant $C$. 

\medskip

\noindent
{\bf Acknowlegement}. The second author is grateful to Michael Larsen for pointing out the square-root upper bound for \eqref{one-dim-sum} for most cases. The first author is grateful to Caltech Student Faculty Programs Office and the Math Department for the Summer Undergraduate Research Fellowship that funded this research. \\

\section{Proof of Theorem \ref{main}}

Square out $\|T (f_1,f_2)\|_2^2$ to get $\|T (f_1,f_2)\|_2^2=$
\begin{align}
\label{pre-1}
    &\sum_{n_1\not=s}\sum_{s,n_2}f_1(s-n_1)f_2(n_1)\overline{f_1(s-n_2)}\overline{f_2(n_2)}K(s-n_1,n_1)\overline{K(s-n_2,n_2)}\\ 
    \label{pre-2}
    &-\sum_{n_1\not=s}\sum_sf_1(s-n_1)f_2(n_1)\overline{f_1(0)}\overline{f_2(s)}K(s-n_1,n_1)\overline{K(0,s)}.
\end{align}
\eqref{pre-1} will be our main term. As for the error term $\eqref{pre-2}$, note that by \eqref{Gauss-sum}, $K(0,s)=0$ unless $s=0$. Hence
\begin{align}
\label{pre-2-sol}
    \eqref{pre-2}=\sum_{n_1\not=0}f_1(-n_1)f_2(n_1)\overline{f_1(0)}\overline{f_2(0)}K(-n_1,n_1)\leq q^{-1/2}\|f_1\|_2^2\|f_2\|_2^2,
\end{align}
which is better than what we claim in Theorem \ref{main}.

\smallskip

As for the main term \eqref{pre-1}, by Cauchy-Schwarz on the variables $n_1,n_2$, 
\begin{equation}
\label{lambda1}
    \eqref{pre-1}\leq \|f_2\|_2^2\cdot|\La_1(f_1)|^{1/2},
\end{equation}
where $\La_1(f)$ is define as 
\begin{equation}
\label{lambda-1-pre}
    \La_1(f)=\sum_{\substack{n_1\not=s_1,\\n_1\not=s_2}}f(s_1-n_1)\overline{f(s_2-n_1)}\overline{f(s_1-n_2)}f(s_2-n_2)H_1(s_1,s_2,n_1,n_2),
\end{equation}
with $H_1$ being given by
\begin{equation}
\label{H1}
    H_1(s_1, s_2, n_1, n_2)=K(s_1-n_1,n_1)\overline{K(s_1-n_2,n_2)}\overline{K(s_2-n_1,n_1)}K(s_2-n_2,n_2).
\end{equation}
We consider two separate cases for the summation of \eqref{lambda-1-pre}: Write $\eqref{lambda-1-pre}=$
\begin{align}
\label{lambda-1-1}
    &\sum_{n_1\not=s}\sum_{s,n_2}f(s-n_1)\overline{f(s-n_1)}\overline{f(s-n_2)}f(s-n_2)H_1(s,s,n_1,n_2)\\ \label{lambda-1-2}
    &+\sum_{\substack{n_1\not=s_1,\\n_1\not=s_2}}\sum_{s_1\not=s_2}f(s_1-n_1)\overline{f(s_2-n_1)}\overline{f(s_1-n_2)}f(s_2-n_2)H_1(s_1,s_2,n_1,n_2).
\end{align}

Again, \eqref{lambda-1-2} will be our main term. As for \eqref{lambda-1-1}, note that if in addition $n_2\not=s$, by \eqref{Gauss-sum} we can bound $H_1$ from above as $|H_1|\leq p^{-2}$; if $n_2=s$, by \eqref{Gauss-sum}, $K(0,s)=0$ unless $s=0$. Hence
\begin{align}
\label{lambdda-1-1-est}
    \eqref{lambda-1-1}&\leq p^{-2}\sum_{n_1\not=s,n_2\not=s}\sum_{s}|f(s-n_1)\overline{f(s-n_1)}\overline{f(s-n_2)}f(s-n_2)|\\ 
    &+p^{-1}\sum_{n_1\not=0}|f(-n_1)\overline{f(-n_1)}\overline{f(0)}f(0)|\lesssim p^{-1}\|f\|_2^2\|g\|_2^2.
\end{align}

\smallskip

Now let us look at the main term \eqref{lambda-1-2}. A side case for \eqref{lambda-1-2} is (take $n_1=s_2$)
\begin{align}
    \sum_{s_1\not=s_2}f(s_1-s_2)\overline{f(0)}\overline{f(s_1-n_2)}f(s_2-n_2)H_1(s_1,s_2,s_2,n_2),
\end{align}
which, since $K(0,s_2)=0$ unless $s_2=0$, equals to
\begin{equation}
\label{equality-side}
    \sum_{s_1\not=0}\sum_{n_2}f(s_1)\overline{f(0)}\overline{f(s_1-n_2)}f(-n_2)H_1(s_1,0,0,n_2).
\end{equation}
Note that $|H_1(s_1,0,0,n_2)|\leq p^{-3/2}$ if in addition $n_2(s_1-n_2)\not=0$. If $n_2(s_1-n_2)=0$, then since $s_1\not=0$, we have either $n_2\not=0$ or $s_1-n_2\not=0$, which gives $|H_1(s_1,0,0,n_2)|\leq p^{-1}$. Therefore, by H\"older's inequality, 
\begin{align}
    \eqref{equality-side}\leq &\,\, p^{-3/2}\sum_{s_1\not=0}\sum_{n_2(s_1-n_2)\not=0}f(s_1)\overline{f(0)}\overline{f(s_1-n_2)}f(-n_2)\\
    &+p^{-1}\sum_{s_1\not=0}\sum_{n_2(s_1-n_2)=0}f(s_1)\overline{f(0)}\overline{f(s_1-n_2)}f(-n_2)\\ \label{last-side-esti}
    \lesssim &\,\, p^{-1}\|f\|_2^2\|g\|_2^2.
\end{align}
Finally, we can express our main term \eqref{lambda-1-2} as
\begin{align}
    &\eqref{lambda-1-2}=-2\cdot \eqref{equality-side}\\ 
    \label{lambda-1-refined}
    &+\sum_{n_1,n_2}\sum_{s_1\not=s_2}f(s_1-n_1)\overline{f(s_2-n_1)}\overline{f(s_1-n_2)}f(s_2-n_2)H_1(s_1,s_2,n_1,n_2).
\end{align}
We will estimate \eqref{lambda-1-refined} in the rest of the paper.

\smallskip

The key observation here is that the four vectors $s_1-n_1$, $s_2-n_1$, $s_1-n_2$, and $s_2-n_2$ only span a three-dimensional space. After a change of variables $x_1=s_1-n_1$, $x_2=s_1-n_2$, $x_3=s_2-n_1$, and $x_4=n_2$ (so $s_1\not=s_2$ is equivalent to $x_1\not=x_3$),
\begin{align}
\nonumber
    \eqref{lambda-1-refined}=\sum_{x_1\not=x_3}&f(x_1)\overline{f(x_3)}\overline{f(x_2)}f(x_2+x_3-x_1)\\ \label{sum-H1}
    &\sum_{x_4}H_1(x_2+x_4,\,x_2+x_3+x_4-x_1,\,x_2+x_4-x_1,\,x_4).
\end{align}
Recall \eqref{H1} for $H_1$ and $\eqref{K}$ for $K$. We expand \eqref{sum-H1} as
\begin{align}
\nonumber
    \eqref{sum-H1}=p^{-4}\sum_{x_4}\sum_{y_1,\ldots, y_4}e_p(Q_1)=p^{-3}\sum_{y_1,y_2,y_3}e_p(R_1),
\end{align}
where $Q_1$ equals to
\begin{equation}
\nonumber
    x_1y_1^2+(x_2+x_4-x_1)y_1-x_2y_2^2-x_4y_2-x_3y_3^2-(x_2+x_4-x_1)y_3+(x_3+x_2-x_1)y_4^2+x_4y_4,
\end{equation}
and 
\begin{equation}
\label{R1}
    R_1=x_1y_1^2-x_2y_2^2-x_3y_3^2+(x_3+x_2-x_1)(y_2+y_3-y_1)^2+(x_2-x_1)(y_1-y_3).
\end{equation}
This gives the simplification 
\begin{align}
\label{lambda1-simplified}
    \eqref{lambda-1-refined}=\sum_{x_1\not=x_3}&f(x_1)\overline{f(x_3)}\overline{f(x_2)}f(x_2+x_3-x_1)K_1(x_1,x_2,x_3),
\end{align}
where
\begin{align}
\label{K1}
    K_1&(x_1,x_2,x_3):=p^{-3}\sum_{y_1,y_2,y_3}e_p(R_1).
\end{align}

\begin{lemma}
\label{Exponential-1}
If $(x_3+x_2)(x_2-x_1)(x_3-x_1)\not=0$ then
\begin{equation}
    |K_1(x_1,x_2,x_3)|\ll p^{-3/2}.
\end{equation}
\end{lemma}
\begin{proof}
Let $(R_{1})_2$ be the homogeneous leading term of $(R_1)_2$. Calculate directly  
\begin{align}
\label{matrix-1}
    \nabla (R_{1})_2/2=
     (y_1,y_2,y_3)\left(\begin{array}{ccc} x_3+x_2 & x_1-x_2-x_3 & x_1-x_2-x_3\\
    x_1-x_2-x_3 & x_3-x_1 & x_3+x_2-x_1\\
    x_1-x_2-x_3 & x_3+x_2-x_1 & x_2-x_1
    \end{array}\right).
\end{align}
Denote by $A$ the 3 by 3 matrix appearing in \eqref{matrix-1}. For $\nabla (R_{1})_2=0$ when $(y_1,y_2,y_3)\not=0$, we need $\det(A)=(x_3+x_2)(x_2-x_1)(x_3-x_1)=0$, which contradicts to our assumption. Hence $\nabla (R_{1})_2\not=0$ when $(y_1,y_2,y_3)\not=0$, and the lemma follows from Theorem \ref{donglisawin-thm}.
\end{proof}

Among other things, lemma \eqref{Exponential-1} gives \eqref{main-esti} for $p^{-1/8}$, which is the exponent obtained in \cite{Dong-Li-Sawin}. To get a better result, one can try to use the oscillation inside the kernel $K_1$, and this is what we are going to do next.

\medskip

By Cauchy-Schwarz on the variables $(x_1, x_3)$, 
\begin{equation}
\label{lambda1-lambda2}
    |\eqref{lambda1-simplified}|\leq \|f\|_2^2 \cdot|\La_2(f)|^{1/2},
\end{equation}
where, with $x_2, x_4$ being two copies of $x_2$, $\La_2(f)$ is defined as
\begin{align}
\label{lambda2}
    \sum_{x_1\not=x_3}f(x_4)\overline{f(x_2)}\overline{f(x_4+x_3-x_1)}f(x_2+x_3-x_1)H_2(x_1,x_2,x_3,x_4)
\end{align}
with $H_2$ being given by 
\begin{equation}
\label{H2}
    H_2(x_1,x_2,x_3,x_4)=K_1(x_1,x_2,x_3)\overline{K_1(x_1,x_4,x_3)}.
\end{equation}

\smallskip

Similar to before, we consider two separate cases in the summation of \eqref{lambda2} ($x_2=x_4$ and $x_2\not=x_4$). Writes $\eqref{lambda2}=$
\begin{align}
\label{H2-case1}
    &
    \sum_{x_1\not=x_3}f(x_2)\overline{f(x_2)}\overline{f(x_2+x_3-x_1)}f(x_2+x_3-x_1)H_2(x_1,x_2,x_3,x_2)\\ \label{H2-case2}
    &+\sum_{\substack{x_1\not=x_3,\\x_2\not=x_4}}f(x_4)\overline{f(x_2)}\overline{f(x_4+x_3-x_1)}f(x_2+x_3-x_1)H_2(x_1,x_2,x_3,x_4).
\end{align}

For the side case \eqref{H2-case1}, note that $H_2(x_1,x_2,x_3,x_2)=|K|^2$, and $x_3-x_1\not=0$ is given already. If $(x_3+x_2)(x_2-x_1)(x_3-x_1)\not=0$, Lemma \ref{Exponential-1} yields $|H_2|\leq p^{-3}$. Plug this back to \eqref{H2-case1} so that\begin{equation}
\label{final1}
    |\eqref{H2-case1}|\leq p^{-2}\|f\|_2^4.
\end{equation}
Suppose $(x_3+x_2)(x_2-x_1)(x_3-x_1)=0$. If $x_3+x_2=0$, 
\begin{align}
\nonumber
    |K_1|&=p^{-3}\Big|\sum_{y_j}e_p(x_1y_1^2-x_2y_2^2+x_2y_3^2-x_1(y_2+y_3-y_1)^2+(x_2-x_1)(y_1-y_3))\Big|\\ \nonumber
    &=p^{-3}\Big|\sum_{y_j}e_p(x_1(y_1+y_3)^2-x_2y_2^2+x_2y_3^2-x_1(y_2-y_1)^2+(x_2-x_1)y_1))\Big|\\ \nonumber
    &\leq p^{-1}.
\end{align}
In the last inequality, we use \eqref{Gauss-sum}, and the fact that $x_1+x_2\not=0$, which is a consequence of $x_3+x_2=0$ and $x_3-x_1\not=0$.

If $x_1-x_2=0$, then
\begin{align}
\nonumber
    |K_1|&=p^{-3}\Big|\sum_{y_j}e_p(x_1y_1^2-x_1y_2^2-x_3y_3^2+
    x_3(y_2+y_3-y_1)^2)\Big|\\ \nonumber
    &=p^{-3}\Big|\sum_{y_j}e_p(x_1y_1^2-x_1y_2^2+x_3(y_2-y_1)^2+2x_3(y_2-y_1)y_3)\Big|\leq p^{-1}.
\end{align}
The last inequality comes from the case study for the cases $x_3(y_2-y_1)\not=0$, $y_2-y_1=0$, and $x_3=0$, where we also need \eqref{Gauss-sum} and $x_3-x_1\not=0$.

Thus, in either case, $|K_1|\leq p^{-1}$, which gives
\begin{equation}
\label{H2-estimate}
    |H_2|\leq p^{-2}.
\end{equation}
Plug this back to \eqref{H2-case1} to get if $(x_3+x_2)(x_1-x_2)=0$, then \begin{equation}
\label{final2}
    |\eqref{H2-case1}|\leq p^{-2}\|f\|_2^4.
\end{equation}

\smallskip

Now let us return to the main case \eqref{H2-case2}. Perform the change of variables $u_1=x_4$, $u_2=x_2$, $u_3=x_3-x_1$, and $u_4=x_1$ so that
\begin{align}
\nonumber
    \eqref{H2-case2}=&\sum_{(u_1-u_2)u_3\not=0}f(u_1)\overline{f(u_2)}\overline{f(u_1+u_3)}f(u_2+u_3)\\ \label{sum-H2}
    &\sum_{u_4}H_2(u_4,u_2,u_3+u_4,u_1).
\end{align}
Again, expand the sum \eqref{sum-H2} to get 
\begin{align}
\nonumber
    \eqref{sum-H2}&=p^{-6}\sum_{u_4}\sum_{y_1,\ldots,y_6}e_p(Q_2)\\  \label{Gauss-sum-app}
    &= p^{-4}\sum_{\substack{y_1,y_3,y_4,y_6\\ 
    G=0}}e_p(R_2).
\end{align}
where $Q_2$ equals to
\begin{align}
\nonumber
    &u_4y_1^2-u_2y_2^2-(u_3+u_4)y_3^3+(u_2+u_3)(y_2+y_3-y_1)^2+(u_2-u_4)(y_1-y_3)\\ \nonumber
    &-[u_4y_4^2-u_1y_5^2+(u_3+u_4)y_6^3+(u_1+u_3)(y_5+y_6-y_4)^2+(u_1-u_4)(y_4-y_6)]\\ \nonumber
    =&\,\, u_4[y_1^2-y_3^2-y_4^2+y_6^2-y_1+y_3+y_4-y_6]\\ \nonumber
    &\,\,+[u_3y_2^2+2(y_3-y_1)(u_2+u_3)y_2]-[u_3y_5^2+2(y_6-y_4)(u_1+u_3)y_5]\\ \nonumber
    &\,\,+u_2y_3^2+(u_2+u_3)y_1^2-2(u_2+u_3)y_1y_3+u_2(y_1-y_3)\\ \nonumber
    &\,\,-[u_1y_6^2+(u_1+u_3)y_4^2-2(u_1+u_3)y_4y_6+u_1(y_4-y_6)],
\end{align}
$R_2$ is given by
\begin{align}
    R_2=&\,-\bar u_3(y_3-y_1)^2(u_2+u_3)^2+\bar u_3(y_6-y_4)^2(u_1+u_3)^2\\ \nonumber
    &\,+u_2y_3^2+(u_2+u_3)y_1^2-2(u_2+u_3)y_1y_3+u_2(y_1-y_3)\\ \nonumber
    &\,-[u_1y_6^2+(u_1+u_3)y_4^2-2(u_1+u_3)y_4y_6+u_1(y_4-y_6)],
\end{align}
and
\begin{equation}
\nonumber
    G=y_1^2-y_3^2-y_4^2+y_6^2-y_1+y_3+y_4-y_6.
\end{equation}
Note that in \eqref{Gauss-sum-app} we use \eqref{Gauss-sum}.

\smallskip

Therefore, we end up with the simplification
\begin{align}
\label{lambda2-simplified}
    \eqref{H2-case2}=\sum_{(u_1-u_2)u_3\not=0}f(u_1)\overline{f(u_2)}\overline{f(u_1+u_3)}f(u_2+u_3)K_2(u_1,u_2,u_3),
\end{align}
where 
\begin{equation}
\label{K2}
    K_2(u_1,u_2,u_3):=p^{-4}\sum_{\substack{y_1,y_3,y_4,y_6\\ 
    G=0}}e_p(R_2).
\end{equation}
\begin{lemma}
\label{exponential-sum-lem2}
Fix $u_1$ and $u_2$. If $(u_1-u_2)u_3\not=0$, then for all but $O(1)$ nonzero $u_3$ we have
\begin{equation}
    |K_2(u_1,u_2,u_3)|=O(p^{-5/2}).
\end{equation}
\end{lemma}

\begin{proof}
We follow the argument in \cite{Dong-Li-Sawin} Section 3. Let $(R_2)_2$ be the homogeneous leading term of $R_2$, and let $G_2$ be the homogeneous leading term of $G$ so that 
\begin{equation}
    \nabla G_2=2(y_1,\, -y_3 ,\, -y_4 ,\, y_6),
\end{equation}
and 
\begin{align}
    \nabla(R_2)_2=\,&2\big(u_2\bar u_3(u_2+u_3)(y_3-y_1)\,,\\ \nonumber
    &u_2\bar u_3(u_2+u_3)y_1+(u_2-(u_2+u_3)^2\bar u_3)y_3 \,,\\  \nonumber
    &u_1\bar u_3(u_1+u_3)(y_4-y_6) \,,\\ \nonumber
    &-u_1\bar u_3(u_1+u_3)y_4-(u_1-(u_1+u_3)^2\bar u_3)y_6 \big).
\end{align}
Fix $(u_1,u_2)$. By Theorem \ref{donglisawin-thm}, it suffices to show that for all by $O(1)$ $u_3$ the matrix $\text{rank}([\nabla G_2]^T, [\nabla(R_2)_2]^T)^T$ has full rank.

Suppose $y_1y_3y_4y_6\not=0$. If $\text{rank}([\nabla G_2]^T, [\nabla(R_2)_2]^T)^T=1$, then the rank of the following matrix
\begin{equation}
\label{matrix2}
    \left(\begin{array}{cccc} y_1 & y_1-y_3 & -y_4 & y_6-y_4\\
    u_2(u_2+u_3)(y_3-y_1) & -u_3^2y_3 & u_1(u_1+u_3)(y_4-y_6) & u_3^2y_6
    \end{array}\right)
\end{equation}
is also 1. Hence the second and the fourth columns yield $y_1/y_3=y_4/y_6$. This together with the first and the third columns give
\begin{equation}
    (u_2(u_2+u_3)-u_1(u_1+u_3))(y_3/y_1-1)=0.
\end{equation}
If $y_3=y_1$, then the first entry of the second column is 0, and hence all $y_1, y_4, (y_6-y_4)$ are zero, which contradicts the assumption $y_1y_3y_4y_6\not=0$. Thus we must have $u_2(u_2+u_3)-u_1(u_1+u_3)=(u_2-u_1)(u_1+u_2+u_3)=0$, which implies $u_1+u_2+u_3=0$ since $u_1\not=u_2$. This proves the lemma when $y_1y_3y_4y_6\not=0$.

\smallskip

If more than two of $\{y_1, y_3, y_4, y_6\}$ are zero, then the Lemma is clearly true. By symmetry, suppose $y_4y_6=0$. If $y_6=0$, then the first and third columns, and the second and third columns of \eqref{matrix2} give
\begin{equation}
    y_1/(y_3-y_1)=u_2(u_2+u_3)/u_1(u_1+u_3), \hspace{.5cm}y_3/(y_3-y_1)=u_1(u_1+u_3)/u_3^2.
\end{equation}
Thus $u_1(u_1+u_3)/u_3^2-u_2(u_2+u_3)/u_1(u_1+u_3)=1$, implying the lemma.

If $y_4=0$, then $\text{rank}([\nabla G_2]^T, [\nabla(R_2)_2]^T)^T=1$ implies the rank of the following matrix 
\begin{equation}
    \left(\begin{array}{ccc} y_1 & y_1-y_3  & y_6\\
    u_2(u_2+u_3)(y_3-y_1) & -u_3^2y_3 & (u_1^2+u_1u_3+u_3^2)y_6
    \end{array}\right)
\end{equation}
has rank 1. Argue similarly as above to get $u_1(u_1+u_3)/u_3^2-u_2(u_2+u_3)/u_1(u_1+u_3)=0$, which implies the lemma again. \qedhere

\end{proof}

Now let us move back to \eqref{lambda2-simplified}. For fixed $u_1,u_2$, let $E(u_1,u_2)$ be the exceptional set of $u_3$ in Lemma \ref{exponential-sum-lem2}, so $|E(u_1,u_2)|=O(1)$. Write
\begin{align}
\label{K2-case1}
    \eqref{lambda2-simplified}=&\sum_{\substack{(u_1-u_2)u_3\not=0,\\u_3\not\in E(u_1,u_2)}}f(u_1)\overline{f(u_2)}\overline{f(u_1+u_3)}f(u_2+u_3)K_2(u_1,u_2,u_3)\\ \label{K2-case2}
    &+\sum_{\substack{(u_1-u_2)u_3\not=0,\\u_3\in E(u_1,u_2)}}f(u_1)\overline{f(u_2)}\overline{f(u_1+u_3)}f(u_2+u_3)K_2(u_1,u_2,u_3)
\end{align}
By Lemma \ref{exponential-sum-lem2}, 
\begin{equation}
    |\eqref{K2-case1}|\leq p^{-5/2}\sum_{u_3}\Big|\sum_{u_1}f(u_1)f(u_1+u_3)\Big|^2\leq p^{-3/2}\|f\|_2^4.
\end{equation}

As for \eqref{K2-case2}, rewrite it as (recall \eqref{H2} and the change of variables we made before: $u_1=x_4$, $u_2=x_2$, $u_3=x_3-x_1$, and $u_4=x_1$) $\eqref{K2-case2}=$
\begin{equation}
\label{K2-case2-rewrite}
    \sum_{\substack{(u_1-u_2)u_3\not=0,\\u_3\in E(u_1,u_2)}}\sum_{u_4}f(u_1)\overline{f(u_2)}\overline{f(u_1+u_3)}f(u_2+u_3)H_2(u_4,u_2,u_3+u_4,u_1).
\end{equation}
Lemma \ref{Exponential-1} state that $|H_2|\ll p^{-3}$ if $u_3(u_2+u_3+u_4)(u_2-u_4)\not=0$. For fixed $(u_1,u_2)$ and $u_3\in E(u_1,u_2)$, there are $O(1)$ $u_4$ that fails $u_3(u_2+u_3+u_4)(u_2-u_4)\not=0$. Denote this set of $u_4$ by $E'(u_1,u_2, u_3)$. Hence
\begin{align}
\nonumber
    \eqref{K2-case2-rewrite}=&\sum_{u_1, u_2}\sum_{\substack{u_3\in E(u_1,u_2),\\u_4\not\in E'(u_1,u_2, u_3)}}f(u_1)\overline{f(u_2)}\overline{f(u_1+u_3)}f(u_2+u_3)H_2(u_4,u_2,u_3+u_4,u_1)\\ \nonumber
    +&\sum_{u_1, u_2}\sum_{\substack{u_3\in E(u_1,u_2),\\u_4\in E'(u_1,u_2, u_3)}}f(u_1)\overline{f(u_2)}\overline{f(u_1+u_3)}f(u_2+u_3)H_2(u_4,u_2,u_3+u_4,u_1)\\ \nonumber
    :=&I+II.
\end{align}
Lemma \ref{Exponential-1} implies $|H_2(u_4,u_2,u_3+u_4,u_1)|\ll p^{-3}$ when $u_4\not\in E'(u_1,u_2, u_3)$. Hence
\begin{equation}
    |I|\leq  p^{-3}\sum_{u_4}\|f\|_2^4\leq p^{-2}\|f\|_2^4.
\end{equation}
As for $II$, \eqref{H2-estimate} implies $|H_2(u_4,u_2,u_3+u_4,u_1)|\leq p^{-2}$ when $u_3\not=0$ (which is guaranteed in the summation of $\eqref{K2-case2-rewrite}$). Hence
\begin{equation}
    |II|\leq  p^{-2}\|f\|_2^4.
\end{equation}

Combining the cases above we finally get
\begin{equation}
    |\eqref{lambda2-simplified}|\ll p^{-3/2}\|f\|_2^4.
\end{equation}
Plug this back to \eqref{lambda1-lambda2} to conclude Theorem \ref{main}.

\section{Ending remarks}

\subsection{} We discard the oscillatory information of $K_2$ in Lemma \ref{exponential-sum-lem2}. To obtain further improvement, a natural attempt is to pick up this information, for example, by applying Cauchy-Schwarz on the variables $(u_1,u_2)$ in \eqref{lambda2-simplified} (this is what we did in \eqref{lambda1-simplified}). In this way, we can similarly get a new kernel $K_3$, which, being parallel to \eqref{sum-H2}, is an exponential sum on a codimension-3 variety with 8 variables. If $|K_3|$ has a square-root-cancellation upper bound, then it is likely to get the exponent $p^{-1/4+1/32}$ for \eqref{main-esti}. One may even hope to get the exponent $p^{-1/4+1/2^{n+3}}$ by iterating the process above for $n$ steps (so that there are kernels $K_j$ up to the $n$-th kernel $K_n$). Of course, the result is conditioned on the assumption that every appearing kernel $K_j$ obeys a square-root-cancellation upper bound.

However, we don't know how to get a square-root-cancellation upper bound for even $K_3$. This is essentially the reason why our argument does not generalize to other kernels $K$ with polynomial phase $aQ(y)+bP(y)$ (recall \eqref{K}). When $(Q,P)=(y^2,y)$, it is quadratic, and we can use the explicit expression of Gauss sum in \eqref{Gauss-sum-app}, while there is no known explicit formula for other polynomials, even for $(Q,P)=(y^3,y)$. As a result, we do not know how to get a square-root-cancellation upper bound for the corresponding kernel.

\subsection{}

It is also possible to get a weaker Theorem \ref{main} without using the very strong higher-dimensional exponential sum result, Theorem \ref{donglisawin-thm}. Indeed, we first make the change of variables $s-n_1\to n_1$ and $s-n_2\to n_2$, then expand the kernel $K$ and express $f_1$ by Fourier transform so that $\|T( f_1, f_2)\|_2^2$ equals to
\begin{align*}
    &p^{-4} \sum_{n_1,n_2} f_1(n_1)\overline{f_1(n_2)} \sum_{\substack{y_1,y_2\\ z_1, z_2 \\ s}} \hat f_2(y_1)\overline{\hat f_2(y_2)}e_p(-(s-n_1)y_1)e_p(y_2(s-n_2)) \\
    & \hspace{4.5cm} \times e_p((s-n_1)z_1+n_1z_1^2) e_p(-(s-n_2)z_2-n_2z_2^2)\\
    &=p^{-4} \sum_{n_1,n_2} f_1(n_1)\overline{f_1(n_2)} \sum_{\substack{y_1,y_2\\ z_1, z_2}}\hat f_2(y_1)\overline{\hat f_2(y_2)} e_p(-(s-n_1)y_1)e_p(y_2(s-n_2))\\
    & \hspace{4.5cm} \times e_p(-n_1z_1+n_1z_1^2) e_p(n_2z_2-n_2z_2^2)\\
    & \hspace{4.5cm} \times \sum_s e_p(s(y_2-y_1+z_1-z_2)).
\end{align*}
For simplicity, we do not consider the minor terms that come from $n_1=0$ or $n_2=0$ but only focus on the main terms. Observe that this sum over $s$ is non-zero only when $y_2+z_1= y_1+z_2$, in which case it equals $p$. Consequently, we may set $z_2=z_1+y_2-y_1$ and eliminate the sum over $z_2$. Apply the change of variables $x_1=y_1$, $x_2=y_2$,and $x_3=z_1$ to yield
\begin{align*}
    &p^{-3}  \sum_{n_1,n_2} f_1(n_1)\overline{f_2(n_2)} \sum_{x_1,x_2,x_3} \hat f_2(x_1) \overline{\hat f_2(x_2)} e_p(n_1x_1)e_p(-x_2n_2)e_p(-n_1x_3+n_1x_3^2) \\
    & \hspace{5cm} \times e_p(n_2(x_2+x_3-x_1)-n_2(x_2-x_1+x_3)^2).
\end{align*}
Now apply the change of variables $n_1-n_2\to n_1$ and $x_2-x_1 \to x_2$ then gather all terms depending on $x_3$ to get
\begin{align*}
    &p^{-3} \sum_{n_1,n_2} f_1(n_1+n_2) \overline{f_1(n_2)} \sum_{x_1,x_2} \hat f_2(x_1) \overline{\hat f_2(x_1+x_2)} e_p(x_1n_1) e_p(-n_2x_2^2) \\
    &\hspace{5cm} \times \sum_{x_3} e_p(x_3(-n_1-2n_2x_2)+n_1x_3^2).
\end{align*}
Employ \eqref{Gauss-sum} to have that the RHS of above equals to
\begin{align*}
    &p^{-5/2} \sum_{n_1,n_2} f_1(n_1+n_2) \overline{f_1(n_2)} \sum_{x_1,x_2} \hat f_2(x_1) \overline{\hat f_2(x_1+x_2)} e_p(x_1n_1) e_p(-n_2x_2^2) \\
        &\hspace{5cm} \pt{\frac{n_1}{p}}e_p(-\overline{4n_1}(n_1+2n_2x_2)^2)\\
    &=p^{-5/2} \sum_{x_1,x_2} \hat f_2(x_1) \overline{\hat f_2(x_1+x_2)} \sum_{n_1,n_2} f_1(n_1+n_2) \overline{f_1(n_2)} e_p(x_1n_1) e_p(-n_2x_2^2)\\
    &\hspace{5cm} \pt{\frac{n_1}{p}}e_p(-\overline{4}n_1-n_2x_2-\overline{n_1}n_2^2x_2^2)\\
    &\leq\frac{\|f_2\|_2^2}{p^{3/2}}  \Big\| \sum_{n_1,n_2} f_1(n_1+n_2) \overline{f_1(n_2)} e_p(x_1n_1) e_p(-n_2x_2^2) \\  \nonumber
    &\hspace{5cm} \pt{\frac{n_1}{p}}e_p(-\overline{4}n_1-n_2x_2-\overline{n_1}n_2^2x_2^2) \Big\|_{\ell^2_{x_1,x_2}}.
\end{align*}

Expand the $\ell^2_{x_1,x_2}$-norm by labeling the mirrors of $n_1,n_2$ as $n_3,n_4$. This yields that the RHS above is
\begin{align*}
    &\frac{\|f_2\|_2^2}{p^{3/2}}   \Bigg(\sum_{x_1,x_2} \sum_{\substack{n_1,n_2\\ n_3,n_4}} f_1(n_1+n_2) \overline{f_1(n_2)}\overline{ f_1(n_3+n_4)} f_1(n_4)  \pt{\frac{n_1}{p}} \pt{\frac{n_3}{p}} \\
    &\hspace{.5cm} \times e_p(x_1(n_1-n_3))e_p(-\overline{4}(n_1-n_3)-x_2(n_2-n_4))e_p((-n_2-\overline{n}_1n_2^2)x_2^2)\\
    &\hspace{.5cm} \times e_p((\overline{n}_3n_4^2+n_4)x_2^2)
    \Bigg)^{1/2}\\
    =&\frac{\|f_2\|_2^2}{p^{3/2}}\Bigg(\sum_{x_2} \sum_{\substack{n_1,n_2\\ n_3,n_4}} f_1(n_1+n_2) \overline{f_1(n_2)}\overline{ f_1(n_3+n_4)} f_1(n_4)  \pt{\frac{n_1}{p}} \pt{\frac{n_3}{p}} \\
    &\hspace{.5cm} \times e_p(-\overline{4}(n_1-n_3)-x_2(n_2-n_4))e_p((-n_2-\overline{n}_1n_2^2)x_2^2)\\
    &\hspace{.5cm} \times e_p((\overline{n}_3n_4^2+n_4)x_2^2) \sum_{x_1}e_p(x_1(n_1-n_3))
    \Bigg)^{1/2}.
\end{align*}
Once again, we observe that the sum over $x_1$ is non-zero only when $n_1=n_3$ where it equals $p$. Since the square of the Legendre symbol is $1$ (note that $n_1\neq 0$ in these sums) we may cancel them, and gathering the $x_2$ terms together yields the RHS above now equals to
\begin{align*}
    &\frac{\|f_2\|_2^2}{p}  \Bigg( p \sum_{n_1,n_2,n_4} f_1(n_1+n_2) \overline{f_1(n_2)} \overline{f_1(n_1+n_4)} f_1(n_4)  \\
        &\hspace{1cm} \times \sum_{x_2} e_p(-x_2(n_2-n_4)+x_2^2 [\overline{n_1}n_4^2+n_4-\overline{n_1}n_2^2-n_2]) \Bigg)^{1/2}.
\end{align*}
Do the change of variable $x_2(n_2-n_4) \to x_2$, $n_3= n_4$ and then evaluate the quadratic Gauss sum to get 
\begin{align*}
     &\frac{\|f_2\|_2^2}{p^{3/4}} \Bigg(  \sum_{n_1,n_2,n_3} f_1(n_1+n_2) \overline{f_1(n_2)} \overline{f_1(n_1+n_3)} f_1(n_3)  \\
    &\hspace{2cm} \pt{\frac{\overline{n_2-n_3} [\overline{n_1}(n_2+n_3)+1]}{p}} e_p(\overline{4} (n_2-n_3) \overline{[\overline{n_1}(n_2+n_3)+1]}) \Bigg)^{1/2}.
\end{align*}
Finally, we apply Cauchy-Schwarz on the $n_2, n_3$ variables so that the term inside the parentheses is bounded above by 
\begin{align*}
     &\|f_1\|_2^2 \Bigg(  \sum_{n_1,n_2,n_3, n_4} f_1(n_1+n_2) \overline{f_1(n_1+n_3)} \overline{f_1(n_4+n_2)} f_1(n_4+n_3)  \\
    &\hspace{.5cm} \pt{\frac{f(n_1,n_2,n_3)g(n_4,n_2,n_3)}{p}} e_p\big(\overline{4} (n_2-n_3)(g(n_1,n_2,n_3)-g(n_4,n_2,n_3)) \big) \Bigg)^{1/2},
\end{align*}
where $g(a,b,c)=\overline{a}(b+c)+1$. Let $y_1=2n_1+n_2+n_3$, $y_2=n_1+n_2+n_3+n_4$, $y_3=n_3-n_2$, $y_4=n_1$ so that the term inside the parentheses equals to 
\begin{align}
\nonumber
     &\sum_{y_1,y_2,y_3} f_1(\bar{2}(y_1-y_3)) \overline{f_1(\bar{2}(y_1+y_3))} \overline{f_1(y_2-\bar{2}(y_1+y_3))} f_1(y_2+\bar{2}(y_1-y_3))  \\ \label{one-dim-sum}
    &\hspace{.5cm} \sum_{y_4}\pt{\frac{h_1(y_1,y_4)h_2(y_1,y_2,y_4)}{p}} e_p\big(\overline{4} y_3(h_2(y_1,y_2,y_4)-h_1(y_1,y_4)) \big),
\end{align}
where $h_1(y_1,y_4)=\overline{y}_4 (y_1-2y_4)+1$ and $h_2(y_1,y_2,y_4)=\overline{(y_2-y_1+y_4)}(y_1-2y_4)+1$.

Note that in \eqref{one-dim-sum}, $y_3$ appears only linearly in the phase function. Hence one can obtain some estimate for \eqref{one-dim-sum} by an $L^2$ method (though \eqref{one-dim-sum}, the sum of $y_4$, is a one-dimensional character sum, and is $O(p^{1/2})$ except when $y_2=y_3=0$ or $y_1=y_2$ and $y_3=0$). This estimate, together with a real interpolation (estimates for an upper-level set and a lower-level set), will give the weaker exponent $7/40$ for Theorem \ref{main}, which however is still an improvement upon \cite{Dong-Li-Sawin}. 

We hope this approach may have potential in other fields lacking good estimates for exponential sums.

\bibliographystyle{plain}
\bibliography{bibli}

\end{document}